\documentclass[reqno,11pt]{amsart}

\IfFileExists{mymtpro2.sty}{%
  \usepackage[subscriptcorrection]{mymtpro2}
}{}

\usepackage{a4,amssymb,physics}

\usepackage{appendix}

{

\newtheorem{theorem}{Theorem}[section]
\newtheorem{proposition}[theorem]{Proposition}
\newtheorem{lemma}[theorem]{Lemma}

\newcommand{\bel}{\begin{equation} \label}
\newcommand{\ee}{\end{equation}}

\newcommand{\supp}{{\mathrm{supp}}}

\newcommand{\C}{{\mathbb C}}
\newcommand{\R}{{\mathbb R}}
\newcommand{\N}{{\mathbb N}}

\newcommand{\Dom}{{\rm D}}

\newcommand{\cC}{{\mathcal C}}

\newcommand{\cF}{{\mathcal F}}
\newcommand{\cH}{{\mathcal H}}

\newcommand{\cS}{{\mathcal S}}
\newcommand{\cU}{{\mathcal U}}

\newcommand{\e}{\mathrm{e}}

\def\beq{\begin{equation}}
\def\eeq{\end{equation}}
\newcommand{\bea}{\begin{eqnarray}}
\newcommand{\eea}{\end{eqnarray}}
\newcommand{\beas}{\begin{eqnarray*}}
\newcommand{\eeas}{\end{eqnarray*}}
\newcommand{\Pre}[1]{\ensuremath{\mathrm{Re} \left( #1 \right)}}

\newcommand{\Schr}{Schr\"odinger }

\newcommand{\labelnummer}{\mbox{\normalfont (\roman{numcount})}}%

\makeatletter

  {\let\curlabelspeicher\@currentlabel%
    \begin{list}{\labelnummer}%
      {\usecounter{numcount}\leftmargin0pt%
        \topsep0.5ex\partopsep2ex\parsep0pt\itemsep0ex\@plus1\p@%
        \labelwidth2.5em\itemindent3.5em\labelsep1em%
      }%
    \let\saveitem\item%
    \def\item{\saveitem%
      \def\@currentlabel{{\upshape\curlabelspeicher}$\,$\labelnummer}}%
    \let\savelabel\label%
    \def\label##1{\savelabel{##1}%
      \@bsphack%
        \ifmmode\else%
          \protected@write\@auxout{}%
          {\string\newlabel{##1item}{{\labelnummer}{\thepage}}}%
        \fi%
      \@esphack%
    }%
  }{\end{list}}%

\renewcommand{\appendix}{\def\thesection{\textsc{Appendix}}}
\begin{document}

\title[Time-fractional quantum diffusion]{Asymptotic analysis of time-fractional quantum diffusion}

\author[P.\ D.\ Hislop]{Peter D.\ Hislop}
\address{Department of Mathematics,
    University of Kentucky,
    Lexington, Kentucky  40506-0027, USA}
\email{peter.hislop@uky.edu}

\author[\'E.\ Soccorsi]{\'Eric Soccorsi}
\address{Aix Marseille Univ, Universit\'e de Toulon, CNRS, CPT, 
Marseille, France}
\email{eric.soccorsi@univ-amu.fr}


\begin{abstract}
We study the large-time asymptotics of the mean-square displacement for the 
time-fractional \Schr equation in $\R^d$. We define the time-fractional derivative by the Caputo derivative and we consider the initial-value problem for the free evolution of wave packets in $\R^d$ governed by the time-fractional \Schr equation $ i^\beta \partial_t^\alpha u = - \Delta u, ~~~~u(t=0) = u_0$, parameterized by two indices $\alpha, \beta \in (0,1]$.
We show distinctly different long-time evolution of the mean square displacement according to the relation between $\alpha$ and $\beta$. In particular, asymptotically ballistic motion occurs only for $\alpha=\beta$. 
\end{abstract}

\maketitle \thispagestyle{empty}

\tableofcontents



\section{Introduction and statement of the problem}\label{sec:introduction}
\setcounter{equation}{0}

The \Schr equation with fractional spatial derivatives has been the object of many studies. It is proposed as a model for anomalous quantum transport. 
In this article, and a companion article \cite{hs1}, we study the effect of replacing the time derivative in the \Schr equation by a fractional time derivative. We study the behavior of the mean-square displacement (MSD) of a wave packet under the time-fractional \Schr equation (TFSE):
\beq\label{eq:tfse1}
i^\beta \partial_t^\alpha u = -\Delta u, ~~~~u(t=0) = u_0,
\eeq
parameterized by two indices $\alpha, \beta \in (0,1]$. We show that the MSD $D_2(u_0,t)$ (see \eqref{eq:msd1}) exhibits different asymptotic behavior depending on these parameters. These models generalize the time-fractional \Schr equations (TFSE) introduced by Naber \cite{naber}, with $\alpha = \beta \in (0,1)$, and by Narahari Achar, Yale, and Hanneken \cite{ayh}, with $\beta = 1$ and $\alpha \in (0,1)$. For our two-parameter generalization of these models, we compute the asymptotics in time of the MSD.  

In particular, we find that the MSD exhibits distinctly different behavior in each of three parameter regimes. Asymptotic ballistic evolution for which $D_2(u_0,t) \sim t^2$ occurs only when $\alpha = \beta$.  A summary of the main results is as follows. Let us consider $0 < \beta \leq 1$ fixed, and take $0 < \alpha \leq 1$. For the first two regimes, we have 
\bea\label{eq:phases1}
0 < \alpha < \beta \leq 1: & D_2(u_0,t) = C_\alpha (u_0) t^{-2 \alpha} + \mathcal{O}(t^{-3 \alpha}) , \label{eq:ab1}  \\
  & & \nonumber \\
  0 < \alpha = \beta \leq 1: &   D_2(u_0,t) = C_\alpha (u_0) t^{2} + \mathcal{O}(t),  \label{eq:ab2} 
  \eea
  where $C_\alpha (u_0) > 0$ is a finite, positive constant whose value may change line-to-line. 
In particular, result \eqref{eq:ab1} shows how the MSD varies for the TFSE with $\beta = 1$ as $\alpha$ varies in $(0,1)$. For this regime, the MSD tends to zero as $t \rightarrow \infty$. On the other hand, \eqref{eq:ab2} shows ballistic behavior of the MSD when the parameters are equal $\alpha = \beta$. The third regime is characterized by exponential growth of the MSD: 
 \bea\label{eq:ab3}   
 0 < \beta < \alpha  \leq 1:  &  &  \\
  e^{2 t \cos (\frac{\pi \beta}{\alpha} ) 
 \Lambda_-^{\frac{2}{\alpha}} } \left( c_\alpha(u_0) + \mathcal{O}(t^{-1}) \right) 
 & \leq ~~ D_2(u_0,t)   & \leq e^{2 t \cos (\frac{\pi \beta}{\alpha} ) 
 \Lambda_+^{\frac{2}{\alpha}} } \left( C_\alpha(u_0) + \mathcal{O}(t^{-1}) \right) ,
  \nonumber 
\eea
where the finite, positive, constants $\Lambda_\pm$ are determined by the initial condition $u_0$, see Proposition \ref{pr1} and \eqref{eq:supp1}. We remark that Naber \cite{naber} stated that he considered the model with $\alpha = \beta$ because the solutions to the TFSE with these parameters behave similarly to the solutions of the \Schr equation. Our result that the MSD is asymptotically ballistic for $\alpha = \beta$ supports this statement.

In our companion article, \cite{hs1}, we studied the effect of replacing $i \partial_t$, in the \Schr equation, by $i^\beta \partial_t^\alpha$, on the time evolution of the edge current of a half-plane quantum Hall model. We proved the existence of a similar transition in the long-time asymptotic behavior of the edge current depending on the relation between $\alpha$ and $\beta$.  

We mention several works on the \Schr equation with time fractional derivatives, such as \cite{bayin, DX, GPP1, GPT1, iomin2009, iomin2011, iomin2019_A, iomin2019_B, laskin2000_1, laskin2000_2, laskin2002, WX}. We refer to \cite[section 1]{hs1} for a description of their contributions and the relation to our work.  We have found the books by Kilbas, Srivastava, and Trujillo \cite{kst}, and by Podlubny \cite{Pod}, to be useful references.

%
%


\subsection{Acknowledgement} The authors thank Yavar Kian for discussions on the topics of this paper.  PDH thanks Aix Marseille Universit\'e for some financial support and hospitality during the time parts of this paper were written. PDH is partially supported by Simons Foundation Collaboration Grant for Mathematicians No.\ 843327. ÉS is partially supported by the Agence Nationale de la Recherche (ANR) under grant ANR-17-CE40-0029.

\section{Well-posedness of the TFSE}\label{sec-wp}
\setcounter{equation}{0}

Let $H_0 := - \Delta$ be the Laplace operator in $\R^d$, $d \ge 1$, with domain $\Dom(H_0):=H^2(\R^d)$, where $H^j (\R^d)$ denotes the Sobolev space of order $j \in \N$. The operator $H_0$ is self-adjoint on this domain in $L^2(\R^d)$. We set $\R_+:=(0,\infty)$. Given $\alpha \in (0,1]$, and $\beta \in (0,1]$, we consider the generalized time-fractional Schr\"odinger equation (TFSE)
\bel{eq1} 
-i^\beta \partial_t^\alpha u(x,t) + H_0 u(x,t) = 0,\ (x,t) \in Q:=\R^d \times \R_+,
\ee
with initial state $u_0$, in an appropriate subspace of $L^2(\R^d)$, see \eqref{eq:initial_cond1}, so that
\bel{eq2}
u(x,0) = u_0(x), ~~\ x \in \R^d. 
\ee
The fractional time derivative $\partial_t^\alpha$, for $\alpha \in (0,1)$, is the Caputo fractional derivative of order $\alpha$ defined by
\beq\label{eq:cap1}
\partial_t^\alpha u(t) :=\frac{1}{\Gamma(1-\alpha)} \int_0^t \frac{u'(s)}{(t-s)^{\alpha}} \dd s,\ u \in W_{\rm loc}^{1,1}(\R_+). 
\eeq
We note that $\partial_t^\alpha u \rightarrow \partial_t u$, as $\alpha \rightarrow 1$,  if, for example, $u \in C^2(\R)$.  This follows from the formula \eqref{eq:cap1} by an integration by parts. 

We define a solution to the system \eqref{eq1}-\eqref{eq2} as any function
$$ u \in L_{\rm loc}^1(\R_+,\Dom(H_0)) \cap W_{\rm loc}^{1,1}(\R_+,L^2(\R^d)), $$
satisfying the two following conditions simultaneously:
\begin{enumerate}
\item  $-i^\beta \partial_t^\alpha u(x,t) + H_0 u(x,t) = 0$ for a.e. $(x,t) \in Q$,
\item $\lim_{t \downarrow 0} \norm{u(\cdot,t)-u_0}=0$,
\end{enumerate}
where $\norm{\cdot}$ denotes the usual norm in $L^2(\R^d)$.

\subsection{Existence and uniqueness result}

Prior to stating the existence and uniqueness result for the solution to \eqref{eq1}-\eqref{eq2}, we introduce some notation. First, we define the Mittag-Leffler function as
\beq\label{eq:ml1}
 E_{\alpha,\gamma}(z):=  \sum_{n=0}^\infty \frac{z^n}{\Gamma(\alpha n + \gamma)},\ \alpha \in (0,1),\ \gamma \in \R,\ z \in \C, 
 \eeq
where $\Gamma$ is the usual Gamma function. We refer to \cite{kst} and \cite{Pod} for  comprehensive discussions of these functions. 
 
Next, we write $\cF$ for the Fourier transform in $\R^d$, i.e.,
$$ 
\cF u(\xi) = \hat{u}(\xi):= (2  \pi)^{-d \slash 2} \int_{\R} e^{-i x \cdot \xi} u(x) \dd x,\ u \in L^2(\R^d),\ \xi \in \R^d, 
$$
where $\cdot$ stands for the Euclidean scalar product in $\R^d$. We recall that the operator $\cF$ is unitary in $L^2(\R^d)$ and we denote its adjoint by $\cF^*$.  We denote by $\cC_0(\R^d)$ the space of compactly supported continuous functions in $\R^d$. Let us introduce the set
\beq\label{eq:initial_cond1} 
\cU_{\alpha,\beta} := \left\{ \begin{array}{cl} \Dom(H_0) & \mbox{if}\ \beta>\alpha \\ \Dom(H_0^{\frac{1}{\alpha}}) & \mbox{if}\ \beta = \alpha \\ \cF^* \cC_0(\R^d) & \mbox{if}\ \beta<\alpha, \end{array} \right. .
\eeq
 Since $\cF^* \cC_0(\R^d) \subset \Dom(H_0^{\frac{1}{\alpha}})$, we notice that $\cU_{\alpha,\beta} \subset \Dom(H_0^{\frac{1}{\alpha}})$ when $\beta \le \alpha$.
We denote by $\norm{\cdot}$  the usual norm in $L^2(\R^d)$, set
$\norm{v}_{\Dom(H_0^{\gamma})}:=\left( \norm{v}^2+ \norm{H_0^{\gamma} v}^2 \right)^{1 \slash 2}$ for all $v \in \Dom(H_0^{\gamma})$, $\gamma \ge 1$, and we equip $\cU_{\alpha,\beta}$ with the norm:
$$ \norm{u}_{\cU_{\alpha,\beta}} := \left\{ \begin{array}{cl} \norm{u}_{\Dom(H_0)} & \mbox{if}\ \beta > \alpha \\ \norm{u}_{\Dom(H_0^{\frac{1}{\alpha}})} & \mbox{if}\ \beta \le \alpha. \end{array} \right.$$

Then, the existence and uniqueness result for the system 
\eqref{eq1}-\eqref{eq2} can be stated as follows.

\begin{proposition}
\label{pr1}
Pick $u_0 \in \cU_{\alpha,\beta}$.
Then, for all $T \in \R_+$, the system \eqref{eq1}-\eqref{eq2} admits a unique solution 
$ u \in  \cC([0,T],\Dom(H_0)) \cap W_{\rm loc}^{1,1}(0,T;L^2(\R^d)), $
which is expressed by
$$ u(x,t) = \cF^* \left( E_{\alpha,1}((-i)^\beta \abs{\cdot}^2 t^\alpha) \hat{u}_0 \right)(x),\ (x,t) \in Q, $$
and there exists a unique positive constant $C$, depending only on $\alpha$, $\beta$ and $d$, such that we have for $\beta \ge \alpha$,
\bel{ee1} 
\norm{u}_{\cC([0,T],D(H_0))} + \norm{u}_{W^{1,1}(0,T;L^2(\R^d))} \le C (1+T) \norm{u_0}_{\cU_{\alpha,\beta}},
\ee
whereas for $\beta < \alpha$,
\bel{ee2}
\norm{u}_{\cC([0,T],\Dom(H_0))} + \norm{u}_{W^{1,1}(0,T;L^2(\R^d))} \le C (1+T) \e^{\cos \left( \frac{\pi \beta}{2\alpha} \right)\Lambda^{\frac{2}{\alpha}} T}
\norm{u_0}_{\cU_{\alpha,\beta}},
\ee
where 
$\Lambda:=\max \{ \abs{\xi},\ \xi \in \supp (\hat{u}_0) \}$.
\end{proposition}

\subsection{Proof of Proposition \ref{pr1}}\label{subsec:pf_prop1}

We start by establishing the two following technical results.
\begin{lemma}
\label{lm1}
Let $t \in \R_+$, let $\xi \in \R^d \setminus \{ 0 \}$ and put
$\kappa:=(-i)^\beta \abs{\xi}^2 t^{\alpha}$. Then, for all $\gamma \in \R$, there exists a constant $C>0$, depending only on $\alpha$ and $\gamma$ such that we have
\bel{e1}
\abs{E_{\alpha,\gamma}(\kappa)} \le \frac{C}{1+\abs{\kappa}},\ \alpha < \beta,
\ee
and
\bel{e2}
{E_{\alpha,\gamma}(\kappa)} \le C\left( (1+\abs{\kappa})^{\frac{1-\gamma}{\alpha}} \e^{\Pre{\kappa^{\frac{1}{\alpha}}}} +\frac{1}{1+\abs{\kappa}} \right),\ \alpha \ge \beta.
\ee
\end{lemma}
\begin{proof}
When $\alpha<\beta$, we pick $\mu \in (\pi \alpha \slash 2, \min(\pi \alpha, \pi \beta \slash 2))$ and 
apply \cite[Theorem 1.6]{Pod}. Since $\abs{\arg(\kappa)} \in [\mu,\pi]$, we get \eqref{e1} directly from \cite[Eq. (1.148)]{Pod}.
Similarly, for $\alpha \ge \beta$, we apply \cite[Theorem 1.5]{Pod} with $\mu \in (\pi \alpha \slash 2, \pi \alpha)$. As $\abs{\arg(\kappa)} \le \mu$, we see that \cite[Eq. (1.147)]{Pod} yields \eqref{e2}.
\end{proof}

\begin{lemma}
\label{lm2}
We have $\Dom(H_0^{\frac{1}{\alpha}}) \subset \Dom(H_0)$ and the embedding is continuous. More precisely, there exists a constant, $C_\alpha>0$, depending only on $\alpha$, such that
$$ \forall u \in \Dom(H_0^{\frac{1}{\alpha}}),\ \norm{u}_{\Dom(H_0)} \le C_\alpha  \norm{u}_{\Dom(H_0^{\frac{1}{\alpha}})}. $$
\end{lemma}
\begin{proof}
Let $v \in \Dom(H_0^{\frac{1}{\alpha}})$. We have $\hat{v}$ in $L^2(\R^d)$ and $\xi \mapsto \abs{\xi}^{\frac{2}{\alpha}} \hat{v}(\xi) \in L^2(\R^d)$. Thus, $\xi \mapsto  \abs{\xi}^2 \abs{\hat{v}(\xi)}^{\alpha} \in L^{\frac{2}{\alpha}}(\R^d)$ and 
$\abs{\hat{v}}^{1-\alpha} \in L^{\frac{2}{1-\alpha}}(\R^d)$. Since 
$$ \frac{1}{\frac{2}{\alpha}} + \frac{1}{\frac{2}{1-\alpha}} = \frac{\alpha}{2}+\frac{1-\alpha}{2}=\frac{1}{2} $$
and 
$$ \abs{\xi}^2 \abs{\hat{v}(\xi)} =  \abs{\xi}^2 \abs{\hat{v}(\xi)}^{\alpha} \abs{v(\xi)}^{1-\alpha},\ \xi \in \R^d, $$
we have $\xi \mapsto \abs{\xi}^2 \hat{v}(\xi) \in L^2(\R^d)$ and
\beq\label{eq:bd1}
\norm{\abs{\xi}^2 \hat{v}(\xi)} \le \norm{\abs{\xi}^2 \abs{\hat{v}(\xi)}^{\alpha}}_{L^{\frac{2}{\alpha}}(\R^d)} 
\norm{\abs{\hat{v}(\xi)}^{1-\alpha}}_{L^{\frac{2}{1-\alpha}}(\R^d)},
\eeq
by H\"older's inequality. Defining $\cF H_0 \cF^{-1}=\hat{H_0}$, the multiplication operator by $| \xi |^2$,  inequality \eqref{eq:bd1}  can be equivalently rewritten as
$$ \norm{\hat{H_0} \hat{v}} \le  \norm{\hat{H_0}^{\frac{1}{\alpha}} \hat{v}}^{\alpha} \norm{\hat{v}}^{1-\alpha}. $$
Applying Young's inequality, we obtain 
$$ \norm{\hat{H_0} \hat{v}} \le  \alpha \norm{\hat{H_0}^{\frac{1}{\alpha}} \hat{v}} + (1-\alpha) \norm{\hat{v}}, $$
which proves the desired result.
\end{proof}

Armed with these two lemmas, we are now in position to prove Proposition \ref{pr1}.

\medskip

\begin{proof}
Let $u$ be a solution to \eqref{eq1}-\eqref{eq2} in the sense of Section \ref{sec-wp}. Then, bearing in mind that 
$\cF H_0 \cF^{-1}=\hat{H_0}$ is the multiplication operator in $L^2(\R^d)$, by the function $\abs{\xi}^2=\xi \cdot \xi$, i.e.,
$$ (\hat{H_0} v)(\xi)= \abs{\xi}^2 v(\xi),\ \xi \in \R^d,\ v \in \Dom(\hat{H_0})= \{ v \in L^2(\R^d),\ \abs{\xi}^2 v(\xi) \in L^2(\R^d) \}, $$
we get upon applying the Fourier transform $\cF$ on both sides of the equations \eqref{eq1} and \eqref{eq2}, that 
\bel{p1}
\left\{ \begin{array}{ll} -i^\beta \partial_t^\alpha \hat{u}(\xi,t) + \abs{\xi}^2 \hat{u}(\xi,t) = 0, & (\xi,t) \in Q \\ \hat{u}(\xi,0)=\hat{u}_0, & \xi \in \R^d, \end{array}
\right. 
\ee
where $\hat{u}(\xi,t):=\left( \cF u(\cdot,t) \right)(\xi)$. 

For each $\xi \in \R^d$, the system \eqref{p1} admits a unique solution 
\bel{p2} 
\hat{u}(\xi,t)=E_{\alpha,1}((-i)^{\beta} \abs{\xi}^2 t^\alpha) \hat{u}_0(\xi),\ (\xi,t) \in Q,
\ee
according to \cite[Theorem 4.3]{kst}. 

\noindent {\it First case: $\alpha<\beta$.} We have 
$$\abs{\hat{H_0}^j \hat{u}(\xi,t)} \le C \abs{(\hat{H_0}^{j} \hat{u}_0)(\xi)},\ j=0,1,\ \xi \in \R^d,\ t \in \R_+, $$
from \eqref{e1} and \eqref{p2}, whence
\bel{q1} 
\norm{\hat{u}(\cdot,t)}_{\Dom(\hat{H_0})} \le C \norm{\hat{u}_0}_{\Dom(\hat{H_0})},\ t \in \R_+,
\ee
and
\bel{q2}
\norm{\hat{u}}_{L^1(0,T;L^2(\R^d))} \le C T \norm{\hat{u}_0}_{L^2(\R^d)}.
\ee
Further, since $\frac{\dd}{\dd z} E_{\alpha,1}(z)=\alpha^{-1} E_{\alpha,\alpha}(z)$ for all $z \in \C$, which follows from \eqref{eq:ml1} by direct computation,  we deduce from \eqref{p2} that
\bel{q3}
\partial_t \hat{u}(\xi,t) = (-i)^\beta \abs{\xi}^2 t^{\alpha-1} E_{\alpha,\alpha}((-i)^{\beta} \abs{\xi}^2 t^\alpha) \hat{u}_0(\xi),\ (\xi,t) \in Q.
\ee
It follows from this and \eqref{e1} that
$$
\abs{\partial_t \hat{u}(\xi,t)} \le C \frac{\abs{\xi}^2 t^{\alpha-1}}{1+\abs{\xi}^2 t^{\alpha}} \abs{\hat{u}_0(\xi)},\ (\xi,t) \in Q.
$$
Thus, integrating with respect to $t$ over $(0,T)$, we get that
\beas
\norm{\partial_t \hat{u}(\xi)}_{L^1(0,T)} & \le &C \ln(1+\abs{\xi}^2 T^\alpha) \abs{\hat{u}_0(\xi)} \\
& \le & C T^{\alpha} \abs{(\hat{H_0} \hat{u}_0)(\xi)},\ \xi \in \R^d,
\eeas
where we used that $\ln(1+s) \le s$ for all $s \ge 0$, and we substituted $C$ for $\alpha^{-1} C$. Therefore, we have
$$ 
\norm{\partial_t \hat{u}}_{L^1(0,T;L^2(\R^d))} \le C T^{\alpha} \norm{\hat{H_0} \hat{u}_0}_{L^2(\R^d)},
$$
and \eqref{ee1}
follows readily from this and \eqref{q1}-\eqref{q2}.\\

\noindent {\it Second case: $\alpha=\beta$.}
This time, it follows from \eqref{e2} and \eqref{p2} that 
\beas
\abs{\hat{H_0}^j \hat{u}(\xi,t)} & \le & C \left( 1 + \frac{1}{1+\abs{\xi}^2 t^\alpha} \right) \abs{(\hat{H_0}^{j} \hat{u}_0)(\xi)} \\
& \le & C \abs{(\hat{H_0}^{j} \hat{u}_0)(\xi)},\ j=0,1,\ (\xi,t) \in Q,
\eeas
where we replaced $2 C$ by $C$ in the last line. Hence the estimates \eqref{q1} and \eqref{q2} are still valid.

Next, with reference to \eqref{q3}, we infer from \eqref{e2} that
\beas
\abs{\partial_t \hat{u}(\xi,t)} & \le &C \abs{\xi}^2 t^{\alpha-1} \left( (1+\abs{\xi}^2 t^\alpha)^{\frac{1-\alpha}{\alpha}} + \frac{1}{1+\abs{\xi}^2 t^\alpha} \right) \abs{\hat{u}_0(\xi)} \\
& \le & C \abs{\xi}^2 t^{\alpha-1} (1+\abs{\xi}^2 t^\alpha)^{\frac{1-\alpha}{\alpha}} \abs{\hat{u}_0(\xi)},\ (\xi,t) \in Q,
\eeas
where we substituted $C$ for $2C$ in the last line. Thus, by integrating with respect to $t$ over $(0,T)$, and then using that $(1+s)^{\frac{1}{\alpha}} \le 2^{\frac{1}{\alpha}} (s^{\frac{1}{\alpha}} + 1)$ for all $s \ge 0$, we obtain that
\beas
\norm{\partial_t \hat{u}(\xi,\cdot)}_{L^1(0,T)} & \le & C \left( (1+\abs{\xi}^2 T^\alpha)^{\frac{1}{\alpha}} - 1 \right) \abs{\hat{u}_0(\xi)} \\
& \le & C \left( 1 + \abs{\xi}^{\frac{2}{\alpha}} T \right) \abs{\hat{u}_0(\xi)},\ 
\xi \in \R^d,
\eeas
where $C$ was substituted for $2^{\frac{1}{\alpha}} C$ in the last line. Therefore, integrating with respect to $\xi$ over $\R^d$, we find that
$$ \norm{\partial_t \hat{u}}_{L^1(0,T;L^2(\R^d))} \le C (1+T) \norm{\hat{u}_0}_{\Dom(\hat{H_0}^{\frac{1}{\alpha}})}.$$
Now, putting this together with \eqref{q1}-\eqref{q2} and Lemma \ref{lm2}, we get \eqref{ee1}.\\

\noindent {\it Third case: $\alpha>\beta$.}
It follows readily from \eqref{e2} and \eqref{p2} that, 
\beas
\abs{\hat{H_0}^j \hat{u}(\xi,t)} & \le & C \e^{\cos \left( \frac{\pi \beta}{2\alpha} \right) \xi^{\frac{2}{\alpha}} t} \abs{(\hat{H_0}^{j} \hat{u}_0)(\xi)} \\
& \le & C  \e^{\cos \left( \frac{\pi \beta}{2\alpha} \right) \Lambda^{\frac{2}{\alpha}} t} \abs{(\hat{H_0}^{j} \hat{u}_0)(\xi)},\ j=0,1,\ (\xi,t) \in \supp(\hat{u}_0) \times \R_+,
\eeas
where we replaced the constant $2C$ by $C$ in the last line.
As a consequence we have
\bel{q6}
\norm{\hat{u}(\cdot,t)}_{\Dom(\hat{H_0})} \le C \e^{\cos \left( \frac{\pi \beta}{2\alpha} \right) \Lambda^{\frac{2}{\alpha}} t} \norm{\hat{u}_0}_{\Dom(\hat{H_0})},\ t \in \R_+.
\ee
Next, with reference to \eqref{q3}, we infer from \eqref{e2} that
\beas
\abs{\partial_t \hat{u}(\xi,t)} & \le &C \abs{\xi}^2 t^{\alpha-1} \left( (1+\abs{\xi}^2 t^\alpha)^{\frac{1-\alpha}{\alpha}} \e^{\cos \left( \frac{\pi \beta}{2\alpha} \right) \abs{\xi}^{\frac{2}{\alpha}} t} + \frac{1}{1+\abs{\xi}^2 t^\alpha} \right) \abs{\hat{u}_0(\xi)} \\
& \le & C \abs{\xi}^2 t^{\alpha-1} (1+ \abs{\xi}^2 t^\alpha)^{\frac{1-\alpha}{\alpha}} \e^{\cos \left( \frac{\pi \beta}{2\alpha} \right) \Lambda^{\frac{2}{\alpha}} t} \abs{\hat{u}_0(\xi)},\ (\xi,t) \in \supp(\hat{u}_0) \times \R_+,
\eeas
where $C$ was substituted for $2C$ in the penultimate line. Thus, by integrating with respect to $t$ over $(0,T)$, and then using that $(1+s)^{\frac{1}{\alpha}} \le 2^{\frac{1}{\alpha}} (s^{\frac{1}{\alpha}} + 1)$ for all $s \ge 0$, we obtain that
\beas
\norm{\partial_t \hat{u}(\xi,\cdot)}_{L^1(0,T)} & \le & C \left( (1+\abs{\xi}^2 T^\alpha)^{\frac{1}{\alpha}} - 1 \right) \e^{\cos \left( \frac{\pi \beta}{2\alpha} \right) \Lambda^{\frac{2}{\alpha}} T} \abs{\hat{u}_0(\xi)} \\
& \le & C \left( 1 + \abs{\xi}^{\frac{2}{\alpha}} T \right) \e^{\cos \left( \frac{\pi \beta}{2\alpha} \right) \Lambda^{\frac{2}{\alpha}} T}\abs{\hat{u}_0(\xi)},\ 
\xi \in \supp(\hat{u}_0),
\eeas
where $2^{\frac{1}{\alpha}} C$ was replaced by $C$. It follows from this, that
$$ \norm{\partial_t \hat{u}}_{L^1(0,T;L^2(\R^d))} \le C (1+T) \e^{\cos \left( \frac{\pi \beta}{2\alpha} \right) \Lambda^{\frac{2}{\alpha}} T}\norm{\hat{u}_0}_{\Dom(\hat{H_0}^{\frac{1}{\alpha}})},$$
which, together with \eqref{q6} and Lemma \ref{lm2}, yields \eqref{ee2} with $\beta<\alpha$.
\end{proof}

\section{Diffusion properties of the TFSE}\label{sec:diffusion1}
\setcounter{equation}{0}

We compute the time-asymptotic behavior of the MSD for solutions of the TFSE with initial conditions in an appropriate subspace of $L^2(\R^d)$. These asymptotics depend on the parameters $\alpha, \beta \in (0,1)$.  

\subsection{Mean square displacement}

For $s \in \R_+$, we define weighted $L^2$-spaces
$$ 
L^{2,s}(\R^d) =L^2(\R^d,(1+\abs{x}^2)^{\frac{s}{2}} \dd x):= \{ v \in L^2(\R^d),\ (1+\abs{x}^2)^{\frac{s}{2}} v \in L^2(\R^d) \},\ s>0. 
$$ 
The mean square displacement (MSD) of $\varphi \in \cC(\overline{\R_+},L^{2,2}(\R^d))$ at time $t \in \overline{\R_+}$, is defined by
\bea\label{eq:msd1}
D_2(\varphi,t) & := & \int_{\R^d} \abs{x}^2 \abs{\varphi(x,t)}^2  \dd x \nonumber \\
& = & \langle \abs{\cdot}^2 \varphi(\cdot,t), \varphi(\cdot,t) \rangle, \label{eq3a}
\eea
where the notation $\langle \cdot , \cdot \rangle$ stands for the usual scalar product in $L^2(\R^d)$ or in $L^2(\R^d)^d$ linear in the first entry. 
As $\cF$ is unitary in $L^2(\R^d)$ and $\cF(\abs{\cdot}^2 \varphi)(\xi,t)=-\Delta \hat{\varphi}(\xi,t)$ for a.e. $\xi \in \R^d$ and all $t \in \overline{\R_+}$, we infer from \eqref{eq3a} upon integrating by parts, that
\begin{eqnarray}
D_2(\varphi,t) 
& = & \langle -\Delta \hat{\varphi}(\cdot,t) , \hat{\varphi}(\cdot,t) \rangle \nonumber \\
& = & \| \nabla \hat{\varphi}(\cdot,t) \|^2.\label{eq3b}
\end{eqnarray}
Since the right hand-side on \eqref{eq3b} is well defined whenever $\hat{\varphi} \in  \cC(\overline{\R_+},H^1(\R^d))$, we introduce the MSD at time $t \in \overline{\R_+}$ of $\varphi \in \cC(\overline{\R_+},L^{2,1}(\R^d))$, as
\bel{eq3c}
D_2(\varphi,t) := \norm{\nabla \hat{\varphi}(\cdot,t)}^2.
\ee

\subsection{Time-fractional quantum diffusion}

\subsubsection{Settings}
Pick $u_0 \in \cH_{\alpha,\beta}$, where we have set
$$ \cH_{\alpha,\beta} := \left\{ \begin{array}{ll} \
\cS(\R^d) & \mbox{if}\ \beta \ge \alpha \\ \cF^* \cC_0^1(\R^d) & \mbox{if}\ \beta<\alpha, \end{array} \right.$$
$\cC_0^1(\R^d)$ being the space of continuously differentiable compactly supported functions in $\R^d$. 
Since it is clear that $\cH_{\alpha,\beta} \subset \cU_{\alpha,\beta}$, the system \eqref{eq1}-\eqref{eq2} admits a unique solution $u$ according to Proposition \ref{pr1}, whose Fourier transform is expressed by
$$
\hat{u}(\xi,t)=E_{\alpha,1}(\kappa) \hat{u}_0(\xi),\ (\xi,t) \in Q,
$$
where we recall that $\kappa=(-i)^{\beta} \abs{\xi}^2 t^\alpha$.
Next, using that $\frac{\dd}{\dd z} E_{\alpha,1}(z)=\alpha^{-1} E_{\alpha,\alpha}(z)$, $z \in \C$, we deduce from \eqref{p2} that
\bel{eq4}
\nabla \hat{u}(\xi,t)=2 \alpha^{-1}(-i)^{\beta} t^\alpha E_{\alpha,\alpha}(\kappa) \xi \hat{u}_0(\xi) + E_{\alpha,1}(\kappa) \nabla \hat{u}_0(\xi),\ (\xi,t) \in Q.
\ee
Thus, in light of Lemma \ref{lm1}, we see that $\hat{u}(\cdot,t) \in H^1(\R^d)$ for all $t \in \overline{\R^+}$, and hence that the MSD of $u$, denoted by $D_2(u_0,t)$ in the sequel, is well defined by \eqref{eq3c} where $\varphi$ is replaced by $u$: Namely, plugging \eqref{eq4} into \eqref{eq3c}, we obtain that
\bel{eq5}
D_2(u_0,t) = \norm{\nabla \hat{u}}^2, t \in \R_+.
\ee

\subsubsection{Asymptotics}

We shall examine the three cases $\alpha<\beta$, $\alpha=\beta$, and $\alpha>\beta$ separately.\\

\noindent {\it First case: $\alpha < \beta$}.
Taking $\mu \in (\pi \alpha \slash 2, \min (\pi \alpha,\pi \beta \slash 2) )$ in \cite[Theorems 1.4]{Pod}, we have $\mu \le \abs{\arg \kappa} \le \pi$, whence
\bel{eq20a}
E_{\alpha,1}(\kappa) = -\frac{1}{\Gamma(1-\alpha)} \kappa^{-1} + O(\abs{\kappa}^{-2}),\ \abs{\kappa} \to \infty,
\ee
and
\bel{eq20b} 
E_{\alpha,\alpha}(\kappa) = -\frac{1}{\Gamma(-\alpha)} \kappa^{-2} + O(\abs{\kappa}^{-3}),\ \abs{\kappa} \to \infty,
\ee
according to \cite[Eq. (1.143)]{Pod}. Thus, plugging $\kappa=(-i)^\beta \abs{\xi}^2 t^{\alpha}$ into \eqref{eq20a}-\eqref{eq20b} and recalling \eqref{eq4}, we infer from \eqref{eq5} that 
\bel{eq25}
D_2(u_0,t)=C_\alpha(u_0) t^{-2 \alpha} + O(t^{-3 \alpha}),\ t \to \infty,
\ee
where 
\bel{eq26}
C_\alpha(u_0):= \norm{\abs{\xi}^{-2} \left( \frac{\nabla \hat{u}_0}{\Gamma(1-\alpha)}  + \frac{2\abs{\xi}^{-2} \xi \hat{u}_0}{\alpha \Gamma(-\alpha)}  \right)}^2.
\ee

\noindent {\it Second case: $\alpha = \beta$}.
Taking $\mu \in (\pi \alpha \slash 2, \min (\pi, \pi \alpha) )$ in \cite[Theorems 1.3]{Pod}, we have $\abs{\arg \kappa} \le \mu$, whence
\bel{eq8a} 
E_{\alpha,1}(\kappa) = \frac{1}{\alpha} e^{\kappa^{\frac{1}{\alpha}}} - \frac{1}{\Gamma(1-\alpha)} \kappa^{-1} + O(\abs{\kappa}^{-2}),\ \abs{\kappa} \to \infty,
\ee
\bel{eq8b} 
E_{\alpha,\alpha}(\kappa) = \frac{1}{\alpha} \kappa^{\frac{1-\alpha}{\alpha}} e^{\kappa^{\frac{1}{\alpha}}} - \frac{1}{\Gamma(-\alpha)} \kappa^{-2} + O(\abs{\kappa}^{-3}),\ \abs{\kappa} \to \infty,
\ee
by \cite[Eq. (1.135)]{Pod}.

Putting \eqref{eq8a}-\eqref{eq8b} into \eqref{eq4}, and using that $\kappa=(-i)^\alpha \abs{\xi}^2 t^\alpha$, we infer from 
\eqref{eq5} that
\bel{eq12} 
D_2(u_0,t) = \frac{4\norm{\abs{\xi}^{\frac{2-\alpha}{\alpha}} \hat{u}_0}^2}{\alpha^4} t^2 + O(t),\ t \to \infty.
\ee

\noindent{\it Third case: $\alpha>\beta$}. 
Taking $\mu$ as in the {\it Second case}, we still have $\abs{\arg(\kappa)} \le \mu$, whence \eqref{eq8a}-\eqref{eq8b} remain valid.
In light of this, and  \eqref{eq4}, this implies the large $t$-asymptotic expansion 
\bea \label{eq:grad3}
\nabla \hat{u}(\xi,t) & = & \e^{\left( \cos \left( \frac{\pi \beta}{2 \alpha} \right) - i \sin \left( \frac{\pi \beta}{2 \alpha} \right) \right) \abs{\xi}^{\frac{2}{\alpha}} t} \left( \frac{2}{\alpha^2} (-i)^{\frac{\beta}{\alpha}} \abs{\xi}^{\frac{2(1-\alpha)}{\alpha}} t \xi \hat{u}_0(\xi) + \frac{1}{\alpha} \nabla \hat{u}_0(\xi) \right)  \nonumber \\
& & - i^\beta \abs{\xi}^{-2} \left(  \frac{2}{\alpha \Gamma(-\alpha)} \abs{\xi}^{-2} \xi  \hat{u}_0(\xi) +  \frac{1}{\Gamma(1-\alpha)}  \nabla \hat{u}_0(\xi) \right) t^{-\alpha}+ O(t^{-2\alpha}). \nonumber \\
 & &  \nonumber  \\ 
\eea
For any $u_0 \in \cF^* C_0^1(\R^d)$, there exist finite constants $\Lambda_\pm \geq 0$, so that   
\beq\label{eq:supp1}
 \supp(\hat{u}_0) \subset \{ \xi \in \R^d,\ \Lambda_- \le \abs{\xi} \le \Lambda_+ \} . 
\eeq
Substituting \eqref{eq:grad3} into \eqref{eq5}, and using $\Lambda_\pm$ defined in \eqref{eq:supp1},  we get the follow bounds as $\ t \to \infty$:
\beas
& & \frac{4}{\alpha^2} \e^{2 \cos \left( \frac{\pi \beta}{2 \alpha} \right)\Lambda_-^{\frac{2}{\alpha}} t} t \left( \norm{\abs{\xi}^{\frac{2-\alpha}{\alpha}} \hat{u}_0}^2 + O(t^{-1}) \right) \\
& \le & D_2(u_0, t) \le \frac{4}{\alpha^2} \e^{2 \cos \left( \frac{\pi \beta}{2 \alpha} \right)\Lambda_+^{\frac{2}{\alpha}} t} t \left( \norm{\abs{\xi}^{\frac{2-\alpha}{\alpha}} \hat{u}_0}^2 + O(t^{-1}) \right),\ t \to \infty.
\eeas
We note that the upper and lower bounds on the MSD  in case 3 grow exponentially in time.

\bigskip


\end{document}